\documentclass[leqno]{amsart}
\usepackage[normal]{boilerart}
\usepackage{mathrsfs}
\usepackage{tikz-cd}
\usepackage{longtable}

\newcommand{\FEt}{\textbf{\textup{FÉt}}}

\newcommand{\cocart}{\textit{cocart}}
\newcommand{\Vect}{\categ{Vect}}

\newcommand{\leftnat}[1]{\vphantom{#1}_{\natural}\mskip-1mu{#1}}
\newcommand{\rightnat}[1]{{#1}^{\natural}}

\newcommand{\ultimes}{\,\underline{\times}\,}
\newcommand{\vop}{\textit{vop}}

\def\reflectedop#1#2{\mathop{\reflectbox{$#1{#2}$}}}

\title[Expos\'e I -- Elements of parametrized higher category theory]{Parametrized higher category theory and higher algebra: Expos\'e I -- Elements of parametrized higher category theory}

\author{Clark Barwick}
\address{Department of Mathematics, Massachusetts Institute of Technology, 77 Massachusetts Avenue, Cambridge, MA 02139-4307, USA}
\email{clarkbar@math.mit.edu}

\author{Emanuele Dotto}
\address{Department of Mathematics, Massachusetts Institute of Technology, 77 Massachusetts Avenue, Cambridge, MA 02139-4307, USA}
\email{dotto@math.mit.edu}

\author{Saul Glasman}
\address{University of Minnesota, School of Mathematics, Vincent Hall, 206 Church St. SE, Minneapolis, MN 55455, USA}
\email{saulglasman0@gmail.com}

\author{Denis Nardin}
\address{Department of Mathematics, Massachusetts Institute of Technology, 77 Massachusetts Avenue, Cambridge, MA 02139-4307, USA}
\email{nardin@math.mit.edu}

\author{Jay Shah}
\address{Department of Mathematics, Massachusetts Institute of Technology, 77 Massachusetts Avenue, Cambridge, MA 02139-4307, USA}
\email{jshah@math.mit.edu}

\begin{document}

\begin{abstract} We introduce the basic elements of the theory of \emph{parametrized $\infty$-categories} and functors between them. These notions are defined as suitable fibrations of $\infty$-categories and functors between them. We give as many examples as we are able at this stage. Simple operations, such as the formation of opposites and the formation of functor $\infty$-categories, become slightly more involved in the parametrized setting, but we explain precisely how to perform these constructions. All of these constructions can be performed explicitly, without resorting to such acts of desperation as straightening. The key results of this Expos\'e are: (1) a universal characterization of the $T$-$\infty$-category of $T$-objects in any $\infty$-category, (2) the existence of an internal Hom for $T$-$\infty$-categories, and (3) a parametrized Yoneda lemma.
\end{abstract}

\maketitle

\tableofcontents

\section{Parametrized $\infty$-categories} Suppose $G$ a finite group. At a minimum, a $G$-$\infty$-category should consist of an $\infty$-category $C$ along with a weak action $\rho$ of $G$. In particular, for every element $g\in G$, one should have an equivalence $\rho(g)\colon\equivto{C}{C}$, and for every $g,h\in G$, one should have a natural equivalence $\rho(gh)\simeq\rho(g)\circ\rho(h)$, and these natural equivalences should then in turn be constrained by an infinite family of homotopies that express the higher associativity of $\rho$.

Following \cite{BarShah}, we will want to encode this data very efficiently as a cocartesian fibration $\fromto{C}{BG}$. However, as we emphasize in the introduction \cite{Exp0}, a $G$-$\infty$-category should contain yet more information. For example, one wishes also to retain the information of the ``honest'' (not homotopy) fixed-point $\infty$-category $C^H$ of $C$ for any subgroup $H\leq G$ along with the conjugation action equivalence $c_g\colon\equivto{C^H}{C^{gHg^{-1}}}$. All of these data should fit together via the obvious restriction functors. These should be data \emph{in addition} to a cocartesian fibration to $BG$.

It is a classical result of Tony Elmendorf and Jim McClure that to give a $G$-equivariant \emph{space}, one needs only to give the data of the honest fixed-point spaces, the residual actions, and their compatibilities with restriction. That is, again falling in line with the approach of \cite{BarShah}, the homotopy theory of $G$-spaces can be identified with the homotopy theory of left fibrations with target the opposite of the orbit category of $G$.

In the context of more general homotopy theories, Marc Stephan \cite{MStephan} and Julie Bergner \cite{Bergequiv} have proved versions of this theorem that exhibit equivalences of homotopy theories between categories enriched in $G$-spaces and cocartesian fibrations with target $\OO_G^{\op}$. So we may simply take the latter as the \emph{definition} of a $G$-equivariant homotopy theory, and that is exactly what we will do.

\begin{dfn}\label{dfn:Gcat} Suppose $G$ a profinite group. Write $\OO_G$ for the ordinary category of transitive, continuous $G$-sets. In particular, its objects are, up to equivalence, the orbits $G/H$ in which the stabilizer $H\leq G$ is an open subgroup.

A \emph{$G$-equivariant $\infty$-category} or, more briefly, a \emph{$G$-$\infty$-category} $C$ is a pair $(C,p)$ consisting of an $\infty$-category $C$ and a cocartesian fibration
\begin{equation*}
p\colon\fromto{C}{\OO_G^{\op}}.
\end{equation*}
The cocartesian fibration $p$ will be called the \emph{structure map for $C$}. From time to time, we will refer simply to $C$ alone as a $G$-$\infty$-category, leaving the structure map implicit.

If $(C,p)$ and $(D,q)$ are two $G$-$\infty$-categories, then a \emph{$G$-functor} $\fromto{(C,p)}{(D,q)}$ is a functor $\fromto{C}{D}$ over $\OO_G^{\op}$ that carries $p$-cocartesian edges to $q$-cocartesian edges.

A \emph{$G$-space} is a $G$-$\infty$-category $(C,p)$ in which every edge of $C$ is  $p$-cocartesian. In particular, the structure map $p$ of a $G$-$\infty$-category $C$ is a left fibration just in case $(C,p)$ is a $G$-space. A $G$-functor between $G$-spaces will simply be called a \emph{$G$-map}.
\end{dfn}

\begin{nul} If a $G$-$\infty$-category $C$ is the nerve of an ordinary category, then this ordinary category is a Grothendieck opfibration that corresponds to a \emph{categorical coefficient system} in the sense of Blumberg--Hill \cite{MR3406512}.
\end{nul}

As we have described in the Introduction, we have cause to contemplate more general bases. This leads us to the following definition.
\begin{dfn} Suppose $T$ an $\infty$-category. Then a \emph{$T$-parametrized $\infty$-category} -- or more briefly, a \emph{$T$-$\infty$-category} -- is a pair $(C,p)$ consisting of an $\infty$-category $C$ and a cocartesian fibration $p\colon\fromto{C}{T^{\op}}$.

Suppose $(C,p)$ and $(D,q)$ two $T$-$\infty$-categories. A \emph{$T$-parametrized functor} -- or more briefly, a \emph{$T$-functor} -- $\fromto{(C,p)}{(D,q)}$ is a functor $\fromto{C}{D}$ over $T^{\op}$ that carries $p$-cocartesian edges to $q$-cocartesian edges.

If $(C,p)$ is a $T$-$\infty$-category whose fibers are all $n$-categories in the sense of \cite[\S 2.3.4]{HTT}, then we will say that $(C,p)$ is a \emph{$T$-$n$-category}.

A \emph{$T$-space} is a $T$-$\infty$-category $(C,p)$ in which every edge of $C$ is $p$-cocartesian. A $T$-functor between $T$-spaces will be called a \emph{$T$-map}.
\end{dfn}

\begin{ntn}\label{ntn:CatinftyT} Suppose $T$ an $\infty$-category. If $(C,p)$ is a $T$-$\infty$-category and $V\in T_0$ is an object, then we write
\[C_V\coloneq p^{-1}\{V\}\]
for the fibre of $p$ over $V$.

The collection of $T$-$\infty$-categories and $T$-functors defines an $\infty$-category $\Cat_{\infty,T}$: precisely, we define $\Cat_{\infty,T}$ as the simplicial nerve of the full simplicial subcategory of $s\Set_{/T^{\op}}^{+}$ spanned by the fibrant objects for the cocartesian model structure.

We also define $\Top_T\subset\Cat_{\infty,T}$ as the full subcategory spanned by the $T$-spaces.

Accordingly, we may employ some notation for marked simplicial sets: for any $T$-$\infty$-cate\-gory $(C,p)$, we write $\leftnat{C}$ for the corresponding marked simplicial set $(C,\iota_{T^{\op}}C)$, where $\iota_{T^{\op}}C\subset C_1$ is the collection of cocartesian edges. If $(C,p)$ and $(D,q)$ are two $T$-$\infty$-categories, then we write
\[\Fun_{T^{\op}}(C,D)\coloneq\Map^{\flat}_{T^{\op}}(\leftnat{C},\leftnat{D}),\]
and we write
\[\Map_{T^{\op}}(C,D)\coloneq\Map^{\sharp}_{T^{\op}}(\leftnat{C},\leftnat{D}).\]
Clearly $\Fun_{T^{\op}}(C,D)$ is an $\infty$-category, and $\Map_{T^{\op}}(C,D)\subset\Fun_{T^{\op}}(C,D)$ is the largest Kan complex contained therein.
\end{ntn}

\begin{exm} When $T=\OO_G$, we write $\Cat_{\infty,G}$, $\Top_{G}$, $\Fun_{G}(C,D)$, and $\Map_{G}(C,D)$ for $\Cat_{\infty,T}$, $\Top_{T}$, $\Fun_{T^{\op}}(C,D)$, and $\Map_{T^{\op}}(C,D)$, respectively.
\end{exm}

\begin{dfn} Suppose $T$ an $\infty$-category. We say that a $T$-functor $\fromto{C}{D}$ is \emph{fully faithful}, \emph{essentially surjective}, or an \emph{equivalence} just in case, for any object $V\in T_0$, the functor $\fromto{C_V}{D_V}$ is so.
\end{dfn}

\begin{nul} Note that a $T$-$\infty$-category $(C,p)$ is classified by an essentially unique functor
\[\CC\colon\fromto{T^{\op}}{\Cat_{\infty}}\]
whose value on an object $V\in T$ is equivalent to the fiber $C_{V}$. If $(C,p)$ is a $T$-space, then $\CC$ lands in the full subcategory $\Top\subset\Cat_{\infty}$. Furthermore, a $T$-functor $\fromto{(C,p)}{(D,q)}$ corresponds to an essentially unique natural transformation $\fromto{\CC}{\DD}$.

In fact, the straightening/unstraightening equivalence yields equivalences of $\infty$-categories
\[\Cat_{\infty,T}\simeq\Fun(T^{\op},\Cat_{\infty})\textrm{\quad and\quad}\Top_{T}\simeq\Fun(T^{\op},\Top).\]
\end{nul}

\begin{nul} We have chosen to define $T$-$\infty$-categories as cocartesian fibrations $p\colon\fromto{D}{T^\op}$ in order maintain certain conventions down the road -- in particular the theory of $\infty$-operads, where cocartesian edges rule the roost. However, this has two disadvantages:
\begin{itemize}
\item Notationally, it can be a bit burdensome to lug around a $T^{\op}$ in subscripts. To address this, we will at times make a global declaration that $S\coloneq T^{\op}$, and use $S$ instead.
\item Additionally, we will at times be presented with a cartesian fibration $q\colon\fromto{D}{T}$, which encodes essentially the same information as a cocartesian fibration $\fromto{C}{T^{\op}}$. In this case, one may apply the dualization construction of \cite{BGN} to $q$ to obtain a cocartesian fibration
\[q^{\vee}\colon\fromto{D^{\vee}}{T^{\op}}\]
that classifies the same functor as $q$.
\end{itemize}
\end{nul}


\section{Examples of parametrized $\infty$-categories} We list a few basic examples of $G$-$\infty$-categories and $T$-$\infty$-categories. We will use these in the sequel.

\begin{exm} Of course $\OO_{G}^{\op}$, which lies over itself via the identity, is a $G$-space. Of course it's the terminal object in the $\infty$-category of $G$-spaces, so we shall write $\ast_G$ for this object.

Similarly, $T^{\op}$ lying over itself via the identity is the terminal $T$-space, and we shall write $\ast_T$ for it.
\end{exm}

\begin{exm} More generally, suppose $C$ an $\infty$-category. Then one has the \emph{constant} $T$-$\infty$-cate\-gory $C\times T^\op$, which lies over $T^\op$ by the projection functor. The functor that classifies it is the constant functor $\fromto{T^\op}{\Cat_\infty}$ at $C$.

Among the constant $T$-$\infty$-categories is the empty $T$-$\infty$-category $\varnothing_T$.
\end{exm}

\begin{exm} If $(C,p)$ and $(D,q)$ are two $T$-$\infty$-categories, then the fiber product
\[C\ultimes D\coloneq(C\times_{T^{\op}}D,(p,q))\]
is the \emph{product} $T$-$\infty$-category.
\end{exm}

\begin{exm}\label{exm:representable} For any object $V\in T_0$, the forgetful functor
\[\fromto{\underline{V}\coloneq(T^{\op})_{V/}}{T^{\op}}\]
exhibits $(T^{\op})_{V/}$ as a $T$-space, which we shall simply denote. The fiber of $\underline{V}$ over an object $W\in T$ is of course $\Map_{T}(W,V)$, and indeed the functor that classifies $\underline{V}$ is the functor represented by $V$. By the Yoneda lemma, for any $T$-$\infty$-category $C$, one has a natural equivalence
\[\Fun_{T^{\op}}(\underline{V},C)\simeq C_V.\]
\end{exm}

\begin{exm} If $T$ is an $\infty$-groupoid (i.e., a Kan complex), then any categorical fibration $\fromto{C}{T^{\op}}$ is automatically a $T$-$\infty$-category, and any $T$-space $\fromto{X}{T^{\op}}$ is a Kan fibration.
\end{exm}

\begin{exm} Suppose $i\colon\fromto{U}{T}$ a functor, and suppose $q\colon\fromto{D}{U^{\op}}$ a $U$-$\infty$-category. Then there exist a $T$-$\infty$-category $p\colon\fromto{C}{T^{\op}}$ and a $U$-functor
\[\fromto{C\times_{T^{\op}}U^{\op}}{D}\]
such that for any object $t\in T_0$, the natural functor
\[\fromto{C_t}{\lim_{s\in U_{t/}}D_s}\]
is an equivalence of $\infty$-categories. The functor that classifies $p$ is right Kan extended from the functor that classifies $q$ along $i$.
\end{exm}

Let's give an explicit construction of the $T$-$\infty$-category $C$ right Kan extended along $i$. We'll need a spot of notation.
\begin{ntn}\label{ntn:laxpb} If $f\colon\fromto{M}{S}$ and $g\colon\fromto{N}{S}$ are two maps of simplicial sets, then let us write
\[M\downarrow_SN\coloneq M\underset{\Fun(\Delta^{\{0\}},S)}{\times}\Fun(\Delta^1,S)\underset{\Fun(\Delta^{\{1\}},S)}{\times}N.\]
A vertex of $M\downarrow_SN$ is thus a vertex $x\in M_0$, a vertex $y\in N_0$, and an edge $\fromto{f(x)}{g(y)}$ in $S_1$. When $M$ and $N$ are $\infty$-categories, the simplicial set $M\downarrow_SN$ is a model for the lax pullback of $f$ along $g$.
\end{ntn}

\begin{exm} Suppose $S$ an $\infty$-category, and suppose $x\in S_0$ a vertex. The simplicial set $\{x\}\downarrow_SS$ is isomorphic to the ``alternative'' undercategory $S^{x/}$ \cite[\S 4.2.1]{HTT}, and, dually, the simplicial set $S\downarrow_S\{x\}$ is isomorphic to the overcategory $S^{/x}$. Consequently, for any two maps $f\colon\fromto{M}{S}$ and $g\colon\fromto{N}{S}$, one has
\[\{x\}\downarrow_SN\cong S^{x/}\times_SN\textrm{\quad and\quad}M\downarrow_S\{x\}\cong M\times_SS^{/x}.\]
\end{exm}

In light of \cite[Prp. 6.5]{BarShah}, we have the following
\begin{prp}\label{prp:RKEofTcats} Suppose $i\colon\fromto{U}{T}$ a functor, and suppose $D$ a $U$-$\infty$-category. Define a functor $\fromto{C}{T^{\op}}$ via the following universal property: for any map $\eta\colon\fromto{K}{T^{\op}}$, we demand a bijection
\[
\Mor_{/T^{op}}(K,C)\cong\Mor_{/U^{\op}}(K\downarrow_{T^{\op}}U^{\op},D),
\]
natural in $\eta$. Then $C$ is a $T$-$\infty$-category, and the natural $U$-functor $\fromto{C\times_{T^{\op}}U^{\op}}{D}$ exhibits $C$ as the right Kan extension of $D$ along $i$.
\end{prp}

\begin{ntn}\label{ntn:finiteTsets} Write $\FF$ for the $1$-category of finite sets. For any $\infty$-category $T$, write $\FF_T$ for the $\infty$-category denoted $\mathscr{P}_{\varnothing}^{\,\FF}(T)$ in \cite[\S 5.3.6]{HTT}. That is, $\FF_T$ is the full subcategory of the $\infty$-category $\Fun(T^\op,\Top)$ spanned by those objects that are equivalent to a finite coproduct of representables. In particular, it enjoys the following universal property: for any $\infty$-category $D$ with all finite coproducts, the Yoneda embedding $j\colon\into{T}{\FF_T}$ induces an equivalence
\[\equivto{\Fun^{\sqcup}(\FF_T,D)}{\Fun(T,D),}\]
where $\Fun^{\sqcup}$ denotes the $\infty$-category of functors that preserve finite coproducts. We call an object of $\FF_T$ a \emph{finite $T$-set}.

The $\infty$-category $\FF_T$ may be described as follows. The objects may be thought of as pairs $(I,X_I)$ consisting of a finite set $I$ and a collection $X_I=\{X_i\}_{i\in I}$ of objects of $T$. The mapping space $\Map_{\FF_T}((J,Y_J),(I,X_I))$ can be identified with the disjoint union
\[\Map_{\FF_T}((J,Y_J),(I,X_I))\simeq\coprod_{\phi\in\Map_{\FF}(J,I)}\prod_{j\in J}\Map_T(Y_j,X_{\phi(j)}).\]

We abuse notation slightly by treating the Yoneda embedding $j\colon\into{T}{\FF_T}$ as if it were an inclusion of a full subcategory. Left Kan extension of the projection $\fromto{T}{\Delta^0}$ along j yields a functor $\mathrm{Orbit}\colon\fromto{\FF_T}{\FF}$ that carries $(I,X_I)$, whence one obtains a decomposition of any finite $T$-set $U$ as
\[U\cong\coprod_{V\in\mathrm{Orbit}(U)}V.\]

Observe that if $T=\OO_G$, then of course $\FF_T=\FF_G$, the category of finite $G$-sets.
\end{ntn}

\begin{exm}\label{discreteGcat} Consider the twisted arrow $\infty$-category $\widetilde{\mathscr{O}}(\FF_T)$ of the $\infty$-category of finite $T$-sets \cite{Aefffib}. This comes equipped with a left fibration
\[\fromto{\widetilde{\mathscr{O}}(\FF_T)}{\FF_T^{\op}\times\FF_T}.\]
Now for any finite $T$-set $U$, one may pull back this left fibration along the inclusion
\[\into{T^{\op}\cong T^{\op}\times\{U\}}{\FF_T^{\op}\times\FF_T}\]
to obtain a $T$-space
\[p_{U}\colon\fromto{\underline{U}}{T^{\op}}.\]
It is easy to see that the functor $\fromto{T^{\op}}{\Top}$ that classifies $p_U$ carries an object $V\in T$ to the space $U(V)\simeq\Map_{\FF_T}(V,U)$. We therefore call this the \emph{discrete $T$-space} attached to $U$.
\end{exm}

The following result is a simple consequence of the Yoneda lemma and the fact that $\underline{U}\simeq\coprod_{V\in\mathrm{Orbit}(U)}\underline{V}$.
\begin{lem}\label{lem:cocartfiboverB} Suppose $(C,p)$ a $T$-$\infty$-category, and suppose $U$ a finite $T$-set. Then the formation of the fibers over each orbit $V\in\mathrm{Orbit}(U)$ induces a trivial fibration
\[\trivfibto{\Fun_{T^{\op}}(\underline{U},C)}{\prod_{V\in\mathrm{Orbit}(U)}C_V}.\]
\end{lem}
\noindent It will from time to time be helpful for us to select a section of this trivial fibration. Of particular importance is the case in which $U$ is terminal:
\begin{dfn} Suppose $(C,p)$ a $T$-$\infty$-category. Then a \emph{cocartesian section of $C$} is a $T$-functor $\fromto{T^{\op}}{C}$.
\end{dfn}

\noindent In particular, if $T$ has a terminal object $1$ (in which case the corresponding left fibration is an equivalence), then the lemma above yields a trivial fibration from the $\infty$-category of cocartesian sections of a $T$-$\infty$-category $C$ and the fiber $\infty$-category $C_{1}$.

\begin{exm}\label{exm:tCatoffiniteTsets} Pull back the target functor
\[\fromto{\mathscr{O}(\FF_T)=\Fun(\Delta^1,\FF_T)}{\Fun(\Delta^{\{1\}},\FF_T)\cong\FF_T}\]
along the inclusion $\into{T}{\FF_T}$. If the $\infty$-category $\FF_T$ admits fiber products -- that is, precisely if $T$ is an \emph{orbital $\infty$-category} in the sense of \cite{Exp3} -- then the projection $\fromto{T\times_{\FF_T}\mathscr{O}(\FF_T)}{T}$ is a cartesian fibration \cite[Lm. 6.1.1.1]{HTT}.

Consequently, we may dualize to obtain the \emph{$T$-$\infty$-category of finite $T$-sets}
\[\fromto{\underline{\FF}_T\coloneq(T\times_{\FF_T}\mathscr{O}(\FF_T))^{\vee}}{T^{\op}}.\]
Note that the fiber over an object $V$ of $T$ is equivalent to the overcategory $(\FF_T)_{/V}$.
\end{exm}

There are other, more sophisticated examples. Here are a few.

\begin{exm} Suppose $E\supseteq F$ a profinite Galois extension of fields with Galois group $G$. Then we can define a $G$-$1$-category $\underline{\Vect}_{E\supseteq F}$ in the following manner. The objects are pairs $(L,X)$ consisting of a subextension $E\supseteq L\supseteq F$ that is finite over $F$ and a finite-dimensional $L$-vector space $X$. A morphism $\fromto{(L,X)}{(L',X')}$ is a field homomorphism $\into{L}{L'}$ over $F$ (but not necessarily under $E$) along with an $L'$-linear map $\fromto{X\otimes_LL'}{X'}$.

The functor $\goesto{(L,X)}{G/\!\Aut_L(E)}$ is the desired structure functor $\fromto{\Vect_{E\supseteq F}}{\OO_G^{\op}}$, which is easily seen to be a Grothendieck opfibration. This was our motivating example from the beginning of the Introduction.
\end{exm}

\begin{exm} We consider the $\infty$-category $\CAlg^{\mathit{cn}}$ of connective $E_{\infty}$ rings and the huge $\infty$-topos relative the regular cardinal $\kappa_1$
\begin{equation*}
\Shv_{\textit{flat}}\subset\Fun(\CAlg^{\mathit{cn}},\Kan(\kappa_1))
\end{equation*}
of large sheaves on $\CAlg^{\mathit{cn},\op}$ for the flat topology \cite[Pr. 5.4]{DAGVII}. We let
\begin{equation*}
\begin{tikzpicture} 
\matrix(m)[matrix of math nodes, 
row sep=5ex, column sep=4ex, 
text height=1.5ex, text depth=0.25ex] 
{\Mod&\QCoh\\ 
\CAlg^{\mathit{cn}}&\Shv_{\textit{flat}}^{\op}\\}; 
\path[>=stealth,->,font=\scriptsize] 
(m-1-1) edge (m-1-2) 
edge node[left]{$q$} (m-2-1) 
(m-1-2) edge node[right]{$p$} (m-2-2) 
(m-2-1) edge[right hook->] (m-2-2); 
\end{tikzpicture}
\end{equation*}
be the pullback square in which $p$ is the right Kan extension of $q$, as constructed in Pr. \ref{prp:RKEofTcats}. The objects of $\QCoh$ can be thought of as pairs $(X,M)$ consisting of a sheaf $X\colon\fromto{\CAlg^{\mathit{cn}}}{\Kan(\kappa_1)}$ for the flat topology and a quasicoherent module $M$ over $X$.

Of course $(\QCoh,p)$ is a $\Shv_{\textit{flat}}$-$\infty$-category and $(\Mod,q)$ is a $\CAlg^{\mathit{cn},\op}$-$\infty$-category, but normally it is awkward to work with such enormous bases. 

To pass to a smaller base, one can consider, for example, a connected, noetherian scheme $X$ and suppose $x$ a geometric point of $X$. Then if $\pi_1^{\et}(X,x)$ is the \'etale fundamental group of $X$, then by Grothendieck's Galois duality \cite[Exp. V, \S 7]{MR50:7129}, the $\infty$-category $\FEt(X)^{\textit{conn}}$ of connected finite \'etale covers is equivalent to $\OO_{\pi_1^{\et}(X,x)}$; then one can pull back $p$ along the forgetful functor
\[\fromto{\OO_{\pi_1^{\et}(X,x)}\simeq\FEt(X)^{\textit{conn}}}{\Shv_{\textit{flat}}}.\]
The result is a $\pi_1^{\et}(X,x)$-$\infty$-category $\QCoh_{\FEt(X)^{\textit{conn}}}$.

As a subexample, we may consider the case in which $X=\Spec F$ with $x=\Spec\Omega$ for a separable closure $\Omega\supset F$ with absolute Galois group $G=\pi_1^{\et}(X,x)$. Then we obtain a derived version of the $G$-$1$-category above.
\end{exm}


\section{Parametrized opposites} Forming the opposite of an $\infty$-category is a trivial matter: one simply precomposes the underlying simplicial set with the unique nontrivial involution $\op\colon\fromto{\Delta}{\Delta}$. For parametrized $\infty$-categories, there is an added wrinkle: one must involve the dualization of \cite{BGN}.

\begin{dfn}\label{dfn:opposite} For any $T$-$\infty$-category $(C,p)$, the \emph{opposite $T$-$\infty$-category} is the $T$-$\infty$-category
\[(C,p)^{\op}\coloneq((C^{\op})^{\vee},(p^{\op})^{\vee})=(({}_{\vee}C)^{\op},({}_{\vee}p)^{\op}),\]
in the notation of \cite{Aefffib}. For brevity, we may write $C^{\vop}$ for $(C,p)^{\op}$. (The ``v'' stands for \emph{vertical}; we think of our construction as a fibrewise opposite construction.) Hence the objects of $C^{\vop}$ are precisely the objects of $C$, and a map $\fromto{x}{y}$ of $C^{\vop}$ is a \emph{cospan}
\begin{equation*}\label{eq:mapfromxtoyinXdual}
\begin{tikzpicture}[baseline]
\matrix(m)[matrix of math nodes, 
row sep=2ex, column sep=3ex, 
text height=1.5ex, text depth=0.25ex] 
{&u&\\ 
y&&x\\}; 
\path[>=stealth,<-,inner sep=0.9pt,font=\scriptsize] 
(m-1-2) edge node[above left]{$g$} (m-2-1) 
edge node[above right]{$f$} (m-2-3); 
\end{tikzpicture}
\end{equation*}
of $C$ in which $p(g)$ is a degenerate edge of $T^{\op}$, and $f$ is a $p$-cocartesian edge.
\end{dfn}

It follows from \cite{BGN} that the assignment $\goesto{(C,p)}{(C,p)^{\op}}$ defines an auto-equivalence of the $\infty$-category $\Cat_{\infty,T}$.


\section{Parametrized subcategories} The notion of a subcategory of an $\infty$-category has a very rigid meaning. In effect, this notion is defined so that ``lying in the subcategory'' is a homotopy invariant notion for morphisms. We introduce the parametrized analogue of that notion.

\begin{rec} Recall \cite[\S 1.2.11]{HA} that a \emph{subcategory} of an $\infty$-category $C$ is a simplicial subset $C'\subset C$ that can be presented as a pullback 
\begin{equation*}
\begin{tikzpicture}[baseline]
\matrix(m)[matrix of math nodes,
row sep=4ex, column sep=4ex,
text height=1.5ex, text depth=0.25ex]
{C' & C \\
N(hC)' & NhC \\ };
\path[>=stealth,->,font=\scriptsize]
(m-1-1) edge[right hook->] (m-1-2)
edge (m-2-1)
(m-1-2) edge (m-2-2)
(m-2-1) edge[right hook->] (m-2-2);
\end{tikzpicture}
\end{equation*}
in which $\into{(hC)'}{hC}$ is the inclusion of a subcategory. Note that this implies that the inclusion $\into{C'}{C}$ is an inner fibration and that $h(C')\simeq(hC)'$.

Properties of subcategories are typically determined by their homotopy categories. For example, we will say that $C'\subset C$ is \emph{stable under equivalences} if $(hC)'\subset hC$ is stable under isomorphisms; in this case the inclusion $\into{C'}{C}$ is a categorical fibration. Similarly, we will say that $C'\subset C$ is \emph{full} if $(hC)'\subset hC$ is full.
\end{rec}

\begin{dfn} Suppose $T$ a $\infty$-category. A \emph{$T$-subcategory} $(C',p')$ of an $T$-$\infty$-cate\-gory $(C,p)$ is a subcategory $C'\subset C$ such that the restriction $p'\coloneq p|C'$ is a cocartesian fibration, and an edge of $C'$ is $p'$-cocartesian if and only if it is $p$-cocartesian.

We will say that $(C',p')\subset(C,p)$ is \emph{stable under equivalences} if a $p$-cocartesian edge $\fromto{x}{y}$ of $C$ lies in $C'$ just in case $x$ does. We will say that $(C',p')\subset(C,p)$ is \emph{full} just in case, for any object $V\in T$, the inclusion $\into{C'_V}{C_V}$ is full.
\end{dfn}

\begin{lem}\label{lem:Gsubcat} Suppose $T$ a $\infty$-category, suppose $(C,p)$ an $T$-$\infty$-category, and suppose $C'\subset C$ a subcategory. Then $C'$ is an $T$-subcategory stable under equivalences just in case the following conditions are satisfied.
\begin{itemize}
\item A $p$-cocartesian edge $\fromto{x}{y}$ of $C$ lies in $C'$ just in case $x$ does.
\item For any commutative triangle
\begin{equation*}
\begin{tikzpicture}[baseline]
\matrix(m)[matrix of math nodes,
row sep=4ex, column sep=4ex,
text height=1.5ex, text depth=0.25ex]
{x&&y\\
&z&\\};
\path[>=stealth,->,font=\scriptsize]
(m-1-1) edge node[above]{$\eta$} (m-1-3)
edge node[below left]{$f$} (m-2-2)
(m-1-3) edge node[below right]{$g$} (m-2-2);
\end{tikzpicture}
\end{equation*}
of $C$ in which $\eta$ is $p$-cocartesian, if $f$ lies in $C'$, then so does $g$.
\end{itemize}
\end{lem}

\begin{lem} Suppose $T$ an $\infty$-category, suppose $(C,p)$ a $T$-$\infty$-category, and suppose $(C',p')$ a $T$-subcategory thereof. Then $(C',p')$ is full in $(C,p)$ if and only if $C'$ is full in $C$.
\begin{proof} Suppose $X,Y\in C'$ are two objects, and write $V$ and $W$ for their images in $T^{\op}$. For any morphism $f\colon\fromto{V}{W}$ in $T^{\op}$, choose a $p'$-cocartesian edge $\fromto{X}{X'}$ lying over $f$ so that the fiber $\Map_{C'}^f(X,Y)$ of
\[\fromto{\Map_{C'}(X,Y)}{\Map_{T^{\op}}(V,W)}\]
over $f$ is equivalent to $\Map_{C'_W}^f(X',Y)$, and the fiber $\Map_{C}^f(X,Y)$ of
\[\fromto{\Map_{C}(X,Y)}{\Map_{T^{\op}}(V,W)}\]
over $f$ is equivalent to $\Map_{C_W}^f(X',Y)$. We thus find that the inclusion map of fibers
\[\fromto{\Map_{C'}^f(X,Y)}{\Map_{C}^f(X,Y)}\]
is an equivalence for any vertex $f\in\Map_{T^{\op}}(V,W)$, whence $\fromto{\Map_{C'}(X,Y)}{\Map_{C}(X,Y)}$ is an equivalence as well.
\end{proof}
\end{lem}

\begin{lem}\label{lem:fullGsubcat} Suppose $T$ an $\infty$-category, suppose $(C,p)$ an $T$-$\infty$-category, and suppose $C'\subset C$ a full subcategory. Then the following are equivalent.
\begin{itemize}
\item $C'$ is a full $T$-subcategory of $C$ that is stable under equivalences.
\item For any $p$-cocartesian edge $\fromto{x}{y}$ of $C$, if $x$ lies $C'$ then so does $y$.
\end{itemize}
\end{lem}


\section{Constructing $T$-$\infty$-categories via pairings} One particularly powerful construction with cartesian and cocartesian fibrations comes from \cite[\S 3.2.2]{HTT}.

\begin{rec} Suppose $p\colon\fromto{X}{T^{\op}}$ a cartesian fibration and $q\colon\fromto{Y}{T^{\op}}$ a cocartesian fibration. Suppose $\XX\colon\fromto{T}{\Cat_{\infty}}$ a functor that classifies $p$ and $\YY\colon\fromto{T^{\op}}{\Cat_{\infty}}$ a functor that classifies $q$. Clearly one may define a functor
\begin{equation*}
\Fun(\XX,\YY)\colon\fromto{T^{\op}}{\Cat_{\infty}}
\end{equation*}
that carries a vertex $s$ of $T^{\op}$ to the $\infty$-category $\Fun(\XX(s),\YY(s))$ and an edge $\eta\colon\fromto{s}{t}$ of $T^{\op}$ to the functor
\begin{equation*}
\fromto{\Fun(\XX(s),\YY(s))}{\Fun(\XX(t),\YY(t))}
\end{equation*}
given by the assignment $\goesto{F}{\YY(\eta)\circ F\circ\XX(\eta)}$.
\end{rec}

\begin{cnstr}\label{cor32213} If one wishes to work instead with the fibrations directly (avoiding straightening and unstraightening), the following construction provides an elegant way of writing explicitly the cocartesian fibration classified by the functor $\Fun(\XX,\YY)$.

Suppose $p\colon\fromto{X}{T^{\op}}$ is a cartesian fibration classified by a functor $\XX\colon\fromto{T}{\Cat_{\infty}}$, and suppose $q\colon\fromto{Y}{T^{\op}}$ is a cocartesian fibration classified by a functor $\YY\colon\fromto{T^{\op}}{\Cat_{\infty}}$. One defines a simplicial set $\widetilde{\underline{\Fun}}_{T^{\op}}(X,Y)$ and a map $r\colon\fromto{\widetilde{\underline{\Fun}}_{T^{\op}}(X,Y)}{T^{\op}}$ defined by the following universal property: for any map $\sigma\colon\fromto{K}{T^{\op}}$, one has a bijection
\begin{equation*}
\Mor_{/T^{\op}}(K,\widetilde{\underline{\Fun}}_{T^{\op}}(X,Y))\cong\Mor_{/T^{\op}}(X\times_{T^{\op}}K,Y),
\end{equation*}
functorial in $\sigma$.

It is then shown in \cite[Cor. 3.2.2.13]{HTT} (see also \cite[Ex. 3.10]{BarShah}) that $r$ is a cocartesian fibration, and an edge
\[g\colon\fromto{\Delta^1}{\widetilde{\underline{\Fun}}_{T^{\op}}(X,Y)}\]
is $r$-cocartesian just in case the induced map $\fromto{X\times_{T^{\op}}\Delta^1}{Y}$ carries $p$-cartesian edges to $q$-cocartesian edges. The fiber of the map $\fromto{\widetilde{\underline{\Fun}}_{T^{\op}}(X,Y)}{S}$ over a vertex $s$ is the $\infty$-category $\Fun(X_s,Y_s)$, and for any edge $\eta\colon\fromto{s}{t}$ of $T^{\op}$, the functor $\eta_!\colon\fromto{T_s}{T_t}$ induced by $\eta$ is equivalent to the functor $\goesto{F}{\YY(\eta)\circ F\circ\XX(\eta)}$ described above.
\end{cnstr}


\section{A technical result: the strong pushforward} Before we proceed to the three main results of this paper, we need a key technical result, Pr. \ref{prp:strongpushforward} We will prove the existence of what we call the \emph{strong pushworward}. In effect, this is an efficient way of selecting compatible families of cocartesian edges in a cocartesian fibration. Glasman has already constructed the strong pushforward in another context. In this section, we give a different, more general construction.

\begin{ntn} In this section, we fix a cocartesian fibration $p\colon\fromto{X}{S}$ of $\infty$-categories. We regard
\[\mathscr{O}(S)\coloneq\Fun(\Delta^1,S)\]
as a cartesian fibration over $S$ via the source map.

Write $X^{\sim}$ for the marked simplicial set $(X,iX)$, where $iX\subset X_1$ is the class of equivalences.
\end{ntn}

\begin{prp}\label{prp:strongpushforward} The cofibration
\[\iota\colon\fromto{X^\sim}{\leftnat{X}\times_{S^\sharp}\rightnat{\mathscr{O}(S)}}\]
of marked simplicial sets is a homotopy equivalence in $s\Set^+_{/S}$.
\end{prp}

The point here is that we can \emph{functorially} select a cocartesian arrow over every arrow of $S$ starting at a given point.

\begin{lem} \label{push} Let $\mathscr{O}^{\cocart}(X)$ denote the full subcategory of $\mathscr{O}(X)$ spanned by the $p$-cocartesian edges. Then the natural map
\[ (\sigma,p)\colon\fromto{\mathscr{O}^{\cocart}(X)}{X\times_S\mathscr{O}(S)}\]
sending every cocartesian arrow $f$ to $(f(0),p(f))$ is a trivial Kan fibration.
\begin{proof} We need to show that for every solid arrow diagram
\begin{equation*}
\begin{tikzpicture}
\matrix(m)[matrix of math nodes,
row sep=4ex, column sep=4ex,
text height=1.5ex, text depth=0.25ex]
{\partial\Delta^n & \mathscr{O}^{\cocart}(X)\\
\Delta^n & X\times_S\mathscr{O}(S),\\ };
\path[>=stealth,->,font=\scriptsize]
(m-1-1) edge (m-1-2)
edge[right hook->] (m-2-1)
(m-1-2) edge node[right]{$(\sigma,p)$} (m-2-2)
(m-2-1) edge (m-2-2)
edge[dotted] (m-1-2);
\end{tikzpicture}
\end{equation*}
there is a dotted lift. To do that we will use the formalism of marked simplicial set. Unwrapping the definitions we see that we need to find a lift on the following diagram of marked simplicial sets
\begin{equation*}
\begin{tikzpicture}[baseline]
\matrix(m)[matrix of math nodes,
row sep=5ex, column sep=3ex,
text height=1.5ex, text depth=0.25ex]
{\left((\Delta^n)^\flat \times \{0\}\right) \cup^{\left((\partial\Delta^n)^\flat \times\{0\}\right)} \left((\partial\Delta^n)^\flat \times (\Delta^1)^\sharp\right) & \leftnat{X} \\
(\Delta^n)^\flat \times (\Delta^1)^\sharp & S^\sharp\\ };
\path[>=stealth,->,font=\scriptsize]
(m-1-1) edge node[above]{} (m-1-2)
edge node[left]{} (m-2-1)
(m-1-2) edge node[right]{} (m-2-2)
(m-2-1) edge node[below]{} (m-2-2)
edge[dotted] (m-1-2);
\end{tikzpicture}
\end{equation*}
By \cite[Pr. 3.1.1.5(2$\mkern-2mu{}^{\prime\prime}$)]{HTT}, the left vertical arrow is the opposite of a marked anodyne arrow, so by \cite[Pr. 3.1.1.6]{HTT} it lifts against all cocartesian fibrations.
\end{proof}
\end{lem}

\begin{cnstr} Let us select a section $c\colon\fromto{X\times_B\mathscr{O}(S)}{\mathscr{O}^{\cocart}(X)}$ of $(\sigma,p)$. Then we can define the \emph{strong pushforward}
\[P_p=\tau\circ c\colon\fromto{X\times_S\mathscr{O}(S)}{X}\]
as the composition of $c$ with the target map from $\mathscr{O}^{\cocart}(X)$ to $X$.
\end{cnstr}

\begin{lem}\label{lem:Ppinverseofiota} The strong pushforward $P_p$ is a marked homotopy inverse of $\iota$.
\begin{proof} First we'll show that $P_p$ is a marked map, and then we will produce the required marked homotopies $P_p\circ\iota\sim \id$ and $\iota\circ P_p\sim \id$.

First, we claim that the strong pushforward $P_p$ is a marked map from $\leftnat{X}\times_{S^\sharp}\rightnat{\mathscr{O}(S)}$ to $X^\sim$. Indeed, choose a marked arrow in $\leftnat{X}\times_{S^\sharp}\rightnat{\mathscr{O}(S)}$. This is the data of a $p$-cocartesian arrow $f\colon\fromto{x}{x'}$ in $X$ and a commutative square
\begin{equation*}
\begin{tikzpicture}[baseline]
\matrix(m)[matrix of math nodes,
row sep=4ex, column sep=4ex,
text height=1.5ex, text depth=0.25ex]
{p(x) & s \\
p(x') & s' \\ };
\path[>=stealth,->,font=\scriptsize]
(m-1-1) edge node[above]{} (m-1-2)
edge node[left]{$p(f)$} (m-2-1)
(m-1-2) edge node[right]{$r$} (m-2-2)
(m-2-1) edge node[below]{} (m-2-2);
\end{tikzpicture}
\end{equation*}
in $S$ such that $r$ is an equivalence. The section $c$ carries this to a square
\begin{equation*}
\begin{tikzpicture}[baseline]
\matrix(m)[matrix of math nodes,
row sep=4ex, column sep=4ex,
text height=1.5ex, text depth=0.25ex]
{x & y \\
x' & y' \\ };
\path[>=stealth,->,font=\scriptsize]
(m-1-1) edge node[above]{} (m-1-2)
edge node[left]{$f$} (m-2-1)
(m-1-2) edge node[right]{$g$} (m-2-2)
(m-2-1) edge node[below]{} (m-2-2);
\end{tikzpicture}
\end{equation*}
in which the horizontal arrows are $p$-cocartesian. But by hypothesis $f$ is also $p$-cocartesian, whence $g$ is as well. Thus since $g$ lies above an equivalence, it follows that $g$ is an equivalence, as desired.
    
Now we construct a marked homotopy between $P_p\circ\iota$ and the identity map from $X^\sim$ to itself. To do this, consider the map $c\circ\iota\colon\fromto{X}{\mathscr{O}^{cocart}(X)\subseteq\mathscr{O}(X)}$. This corresponds to a map
\[\fromto{X\times\Delta^1}{X},\]
which is easily seen to be a marked map $\fromto{X^\sim\times(\Delta^1)^\sharp}{X^\sim}$. This defines our homotopy between the identity map and $P_p\circ\iota$.
    
Finally, we construct a marked homotopy between $\iota\circ P_p$ and the identity map from $\leftnat{X}\times_{S^\sharp}\rightnat{\mathscr{O}(S)}$ to itself. We will construct a functor
\[ \fromto{X\times_S\mathscr{O}(S)}{\Fun(\Delta^1,X\times_S\mathscr{O}(S))\cong\Fun(\Delta^1,X)\times_{\Fun(\Delta^1,S)}\Fun(\Delta^1,\mathscr{O}(S))}\]
realizing the required homotopy. The first component is just given by $c$ via
\[ c\colon\fromto{X\times_S\mathscr{O}(S)}{\mathscr{O}^{cocart}(X) \subseteq \mathscr{O}(X)}.\]
The second component is obtained by composing the second projection with a map
\[\fromto{\mathscr{O}(S)}{\Fun(\Delta^1,\mathscr{O}(S))\cong\Fun(\Delta^1\times\Delta^1,S)}\]
induced by the map $\min\colon\fromto{\Delta^1\times\Delta^1}{\Delta^1}$, which sends $(0,0)$ to 0 and the other objects to 1. Our functor thus sends an object $(x,f\colon\fromto{p(x)}{b})\in X\times_S\mathscr{O}(S)$ to the pair
\begin{equation*}
\left(c(x)\colon\fromto{x}{y},\;
\begin{tikzpicture}[baseline]
\matrix(m)[matrix of math nodes,
row sep=4ex, column sep=4ex,
text height=1.5ex, text depth=0.25ex]
{p(x) & s \\
s & s \\ };
\path[>=stealth,->,font=\scriptsize]
(m-1-1) edge node[above]{$f$} (m-1-2)
edge node[left]{$f$} (m-2-1)
(m-1-2) edge[-,double distance=1.5pt] (m-2-2)
(m-2-1) edge[-,double distance=1.5pt] (m-2-2);
\end{tikzpicture}
\right).
\end{equation*}
One verifies easily that this defines a marked map
\[\fromto{(X^\sim\times_{S^\sharp}\rightnat{\mathscr{O}(S)})\times(\Delta^1)^\sharp}{X^\sim\times_{S^\sharp}\rightnat{\mathscr{O}(S)}}\]
and therefore the desired homotopy.\qedhere
\end{proof}
\end{lem}


\section{$T$-objects in $\infty$-categories} We can talk about $G$-objects in an arbitrary $\infty$-category in in exactly the manner proposed by Elmendorf: the $\infty$-category of $G$-objects in an $\infty$-category $D$ can simply be defined as the $\infty$-category of functors $\fromto{\OO_G^\op}{D}$. Similarly, the $\infty$-category of $T$-objects in an $\infty$-category $D$ can simply be defined as the $\infty$-category of functors $\fromto{T^\op}{D}$.

We will go further in this section, and define a $G$-$\infty$-category of $G$-objects in an $\infty$-category, whose $H$-fixed point $\infty$-category is the $\infty$-category of $H$-objects. Similarly, we define a $G$-$\infty$-category of $G$-objects in an $\infty$-category, whose fibre over an object $V$ is the $\infty$-category of $T_{/V}$-objects. Furthermore, we will prove Th. \ref{thm:univpropDG} that this $T$-$\infty$-category enjoys a useful universal property: if effect, it is the right adjoint to the Grothendieck construction $\goesto{(C,p)}{C}$ (i.e., the oplax colimit).

Put differently, the $T$-$\infty$-category of $T$-objects in $D$ is the cofree $T$-$\infty$-category cogenerated by $D$. This fact, which appears to be quite well known in ordinary category theory, is new in the $\infty$-categorical context.

\begin{dfn} Suppose $D$ an $\infty$-category. The category $D_G$ of \emph{$G$-objects in $D$} is simply the functor $\infty$-category $\Fun(\OO_G^\op, D)$.

More generally, for any $\infty$-category $T$, the category $D_T$ of \emph{$T$-objects in $D$} is the functor $\infty$-category $\Fun(T^\op, D)$.
\end{dfn}

\begin{exm} When $D$ is $\Cat_{\infty}$ or $\Top$, the straightening/unstraightening equivalences justify this notation. Indeed, they can be rewritten in the following manner:
\[\Cat_{\infty,T}\simeq(\Cat_{\infty})_T\textrm{\quad and\quad}\Top_{T}\simeq(\Top)_T.\]
\end{exm}

\begin{wrn} Note, however, that the category $\FF_G$ of finite $G$-sets is \emph{not} the $\infty$-category of $G$-objects in the category of finite sets. 
\end{wrn}

It turns out that the $\infty$-category $D_G$ is in fact the fiber over $G/G$ of a $G$-$\infty$-category. We proceed to define this $G$-$\infty$-category.
\begin{dfn}\label{dfn:TinftycatofTobjects} Suppose $T$ an $\infty$-category, and suppose $D$ an $\infty$-category. The source functor
\[\fromto{\Fun(\Delta^1,T^\op)}{\Fun(\Delta^{{\{0\}}},T^\op)\cong T^\op}\]
is a cartesian fibration. We therefore use Cnstr. \ref{cor32213} to define a simplicial set $\underline{D}_T$ over $T^\op$ thus:
\[\underline{D}_T\coloneq\widetilde{\underline{\Fun}}_{T^\op}(\Fun(\Delta^1,T^\op),D\times T^\op).\]
The structure morphism $\underline{D}_T\to T^\op$ is thus a cocartesian fibration, whence it exhibits $\underline{D}_T$ as a $T$-$\infty$-category. We'll refer to $\underline{D}_T$ as the \emph{$T$-$\infty$-category of $T$-objects in $D$}.

When $T=\OO_G$, we will write $\underline{D}_G$ for $\underline{D}_T$.
\end{dfn}

\begin{exm} Suppose $T$ an $\infty$-category. Taking $D=\Cat_{\infty}$ in Df. \ref{dfn:TinftycatofTobjects}, we obtain the \emph{$T$-$\infty$-category of $T$-$\infty$-categories} $\underline{\Cat}_{\infty,T}$. Similarly, we obtain the \emph{$T$-$\infty$-category of $T$-spaces} $\underline{\Top}_{T}$.
\end{exm}

\begin{nul} Suppose $T$ an orbital $\infty$-category, and suppose $D$ an $\infty$-category. The fiber $(\underline{D}_T)_V$ of $\underline{D}_T$ over an object $V\in T$ is the $\infty$-category of functors
\[\fromto{\underline{V}}{D},\]
giving an equivalence $(\underline{D}_T)_V\simeq D_{T_{/V}}$.

In particular, for any $G$-orbit $G/H$, the fiber $(\underline{D}_G)_{(G/H)}$ of $\underline{D}_G$ is the $\infty$-category of functors
\[\fromto{\underline{G/H}}{D}.\]
An equivalence $\OO_H\simeq(\OO_G)_{/(G/H)}$ specifies an equivalence $(\underline{D}_G)_{(G/H)}\simeq D_H$.
\end{nul}

\begin{nul} Amusingly, one recovers the Bockstein map in an unexpected manner here. Indeed, suppose $G$ a profinite group and $H\leq G$ an open subgroup. Then the full subcategory of $\OO_G^\op$ spanned by $G/H$ is the classifying space $B(N_GH/H)$.

If we restrict the cocartesian fibration
\[\underline{D}_G \to \OO_G^\op\]
to this full subcategory, we obtain, for any $\infty$-category $D$, an action of $N_GH/H$ on the fibre $(\underline{D}_G)_{(G/H)}$. By the Yoneda lemma, this action is induced by an $(N_GH/H)$-action on $\underline{G/H}$ itself; this action is given by restricting the cartesian fibration
\[\Fun(\Delta^1,\OO_G^\op) \to \Fun(\Delta^{\{0\}},\OO_G^\op)\simeq\OO_G^\op\]
to $B(N_GH/H)$.

The full subcategory of the fiber product
\[\Fun(\Delta^1,\OO_G^\op) \times_{ \OO_G^\op} B(N_GH/H)\]
spanned by the object $G/H \to G/1$ is equivalent to the classifying space $BN_GH$, and if we restrict to this full subcategory, the resulting right fibration $BN_GH \to B(N_GH/H)$ is induced by the quotient homomorphism $N_GH \to N_GH/H$. Now under the equivalence
\[\Top_{/ B(N_GH/H)} \simeq \Fun(B(N_GH/H), \Top),\]
this assertion corresponds to the observation that restricting the action of $G/H$ on $\OO_H^\op$ to the full subcategory $BH$ spanned by the object $H/1$ recovers the Bockstein map
\[B(N_GH/H) \to B \Aut(B H).\]
\end{nul}

We now turn to the universal property of the $G$-$\infty$-category of $G$-objects (or, more generally, the $T$-$\infty$-category of $T$-objects). In effect, we will show that the the $G$-$\infty$-category of $G$-objects in an $\infty$-category $D$ is the cofree $G$-$\infty$-category cogenerated by $D$.

\begin{thm}\label{thm:univpropDG} Suppose $T$ an $\infty$-category, $C$ a $T$-$\infty$-category, and $D$ an $\infty$-category. Then there is a natural equivalence
\[\Fun_{T^\op}(C,\underline{D}_T)\simeq\Fun(C,D).\]
\begin{proof} Since the cocartesian arrows in $\underline{D}_T$ are those for which the map
\[\fromto{\Delta^1\times_{T^\op}\Fun(\Delta^1,T^\op)}{D}\]
sends cartesian arrows to equivalences, we can identify $\Fun_{T^\op}(C,\underline{D}_T)$ with the $\infty$-category
\[\Map^\flat(\leftnat{C}\times_{(T^\op)^\sharp}\rightnat{\Fun(\Delta^1,T^\op)},D^\sim).\]

Note that the map $\fromto{T^\op}{\Fun(\Delta^1,T^\op)}$ sending every object to its identity arrow produces a natural cofibration of marked simplicial sets
\[\iota\colon\into{C^\sim}{\leftnat{C}\times_{(T^\op)^\sharp}\rightnat{\Fun(\Delta^1,T^\op)}}.\]
This in turns yields a natural map
\begin{multline*}    
\iota^{\star}\colon\Fun_{T^\op}(C,\underline{D}_T)\cong\Map^\flat(\leftnat{C}\times_{(T^\op)^\sharp}\rightnat{\Fun(\Delta^1,T^\op)},D^\sim)\to\\
\to\Map^\flat(C^\sim,D^\sim)\cong\Fun(C,D).
\end{multline*}

The map $\iota$ is a \emph{trivial} cofibration for the cartesian model structure on the category of marked simplicial sets (see Pr. \ref{prp:strongpushforward}), whence $\iota^{\star}$ is a trivial fibration. This completes the proof.
\end{proof} 
\end{thm}

If now $D$ is a presentable $\infty$-category such that finite products in $D$ preserve colimits separately in each variable, then $D_G$ admits internal Homs. We now show that these internal Homs admit a convenient description. A technical lemma is required.

\begin{lem} \label{muik} Suppose $T$ an $\infty$-category, and suppose $D$ a presentable $\infty$-category such that finite products in $D$ preserve colimits separately in each variable. For any object $V\in T$, let
\[i_V \colon \underline{V} \to T^\op\]
denote the structure morphism for the $T$-space represented by $V$, and consider the adjunction
\[\adjunct{i_{V,!}}{D_{T_{/V}}}{D_T}{i_V^{\star}}.\]
For any $G$-object $X$ of $D$, the morphism
\[i_{V,!}i_V^{\star} X \to X \times i_{V,!} (\ast)\]
of $D_T$ given by the counit of the adjunction on the first factor and $i_{V,!}$ of the projection from $i_V^{\star} X$ to the terminal object $t$ on the second is an equivalence.
\begin{proof} For any object $W\in T$, the claim is that the natural transformation from the composite
\begin{equation*}
(T^{\op})_{V/}\times_{T^\op}(T^\op)_{/W}\ \tikz[baseline]\draw[>=stealth,->,font=\scriptsize](0,0.5ex)--(0.5,0.5ex);\ (T^{\op})_{V/}\ \tikz[baseline]\draw[>=stealth,->,font=\scriptsize](0,0.5ex)--node[above]{$i_V$}(0.65,0.5ex);\ T^{\op}\ \tikz[baseline]\draw[>=stealth,->,font=\scriptsize](0,0.5ex)--node[above]{$X$}(0.5,0.5ex);\ D
\end{equation*}
to the constant functor $\fromto{(T^{\op})_{V/}\times_{T^\op}(T^\op)_{/W}}{D}$ at the object $X(W)$ induces an equivalence on colimits.

Observe that the full subcategory $K$ of $(T^{\op})_{V/}\times_{T^\op}(T^\op)_{/W}$ spanned by objects of the form
\[(\pi\colon\fromto{V}{U},\rho\colon\equivto{U}{W})\]
where $\rho$ is an isomorphism is cofinal. This full subcategory $K$ is an $\infty$-groupoid, and the result follows from the identification
\[K\otimes X(W)\simeq(K\otimes\ast)\times X(W),\]
since the product preserves colimits separately in each variable.
\end{proof}
\end{lem}

\begin{ntn} Suppose $T$ an $\infty$-category, and suppose $D$ a presentable $\infty$-category such that finite products in $D$ preserve colimits separately in each variable. Denote by $F_T(-, -)$ the internal mapping object in $D_T$, so that $F_T(X,-)$ is right adjoint to $X\times -$ for any object $X\in D_T$. Denote by $F_{D_T}(-,-)$ the internal mapping object in $D_T$.
\end{ntn}

\begin{prp}\label{} Suppose $T$ an $\infty$-category, and suppose $D$ a presentable $\infty$-category such that finite products in $D$ preserve colimits separately in each variable. For each object $V\in T$ and each pair $(X, Y)$ of $T$-objects of $D$, we have an equivalence
\[F_T(X, Y)(V) \simeq F_{D_{T_{/V}}}(i_V^{\star} X,i_V^{\star} Y).\]
\end{prp}
\begin{proof} Since $V$ is an initial object of $(T^\op)_{V/}$, the constant functor $\ast \colon T^\op \to D$ is left Kan extended from $\{V\}$, and we have
\[F_{D_{T_{/V}}}(\ast,i_V^{\star} Z) \simeq Z(V).\]
for $Z \in D_T$. So
\begin{align*}
F_T(X, Y)(V) & \simeq F_{D_{T_{/V}}}\left(\ast,i_V^{\star} F_T(X, Y)\right) \\
&\simeq F_{D_T}\left(i_{V,!}(\ast),F_T(X,Y)\right) \\
& \simeq F_{D_T}\left(X \times i_{V,!}(\ast), Y\right) \\
& \simeq F_{D_T}\left(i_{V,!}i_V^{\star}X, Y\right) & \text{(by \ref{muik})} \\
& \simeq F_{D_{T_{/V}}}\left(i_V^{\star}X,i_V^{\star}Y\right),
\end{align*}
and this completes the proof.
\end{proof}


\section{Parametrized fibrations} The universal property of $T$-objects in an $\infty$-category can be used to formulate an equivariant version of the usual straightening/unstraightening constructions.

\begin{nul} Observe that if $\fromto{C}{T^{\op}}$ is a $T$-$\infty$-category, and if $\fromto{D}{C}$ is cocartesian fibration, then $D$ automatically inherits the structure of a $T$-$\infty$-category.
\end{nul}

\begin{dfn} Suppose $T$ an $\infty$-category, and suppose $C$ and $D$ two $T$-$\infty$-categories. Then a \emph{cocartesian $T$-fibration} $\fromto{C}{D}$ is nothing more than a $T$-functor that is also a cocartesian fibration. Likewise, a \emph{left $T$-fibration} $\fromto{C}{D}$ is nothing more than a $T$-functor that is also a left fibration.
\end{dfn}

Combining Th. \ref{thm:univpropDG} with the usual straightening/unstraightening equivalence, we obtain
\begin{prp}\label{prp:Tstraighten} Suppose $T$ an $\infty$-category, and suppose $C$ a $T$-$\infty$-category. Then there are equivalences of $\infty$-categories
\[\Fun_{T^{\op}}(C,\underline{\Cat}_{\infty,T})\simeq\Cat^\cocart_{\infty,/C}\text{\quad and\quad}\Fun_{T^{\op}}(C,\underline{\Top}_{T})\simeq\Cat^{\textit{left}}_{\infty,/C},\]
where $\Cat^\cocart_{\infty,/C}$ denotes the subcategory of $\Cat_{\infty,/C}$ whose objects are cocartesian fibrations and whose morphisms presreve cocartesian edges, whereas $\Cat^{\textit{left}}_{\infty,/C}$ denotes the full subcategory of $\Cat_{\infty,/C}$spanned by the left fibrations.
\end{prp}

We may therefore speak of cocartesian $T$-fibrations over a $T$-$\infty$-category $C$ as being \emph{classified by} a $T$-functor $\fromto{C}{\underline{\Cat}_{\infty,T}}$ and, similarly, of left $T$-fibrations over a $T$-$\infty$-category $C$ as being \emph{classified by} a $T$-functor $\fromto{C}{\underline{\Top}_{T}}$.


\section{Parametrized functor categories} It is important to know that parametrized $\infty$-categories have an internal Hom. More precisely, one wishes to know that for any two $T$-$\infty$-categories $X$ and $Y$, there is a $T$-$\infty$-category $\underline{\Fun}_{T^{\op}}(X,Y)$ that enjoys the universal property
\[\Fun_{T^{\op}}(Z\ultimes X,Y)\simeq\Fun_{T^{\op}}(Z,\underline{\Fun}_{T^{\op}}(X,Y))\]
for any $T$-$\infty$-category $Z$.

\begin{rec} If $p\colon\fromto{X}{T^{\op}}$ and $q\colon\fromto{Y}{T^{\op}}$ are two $T$-$\infty$-categories that are classified by functors
\[\XX,\YY\colon\fromto{T^{\op}}{\Cat_\infty},\]
then it is easy to see that the product functor $\XX\times\YY$ classifies the cocartesian fibration
\[(p,q)\colon\fromto{X\times_{T^{\op}}Y}{T^{\op}}.\]
It's also apparent that the product admits a right adjoint; indeed, the functor $\infty$-category $\Fun(T^{\op},\Cat_\infty)$ admits an internal Hom given by the functor
\[\goesto{V}{\mathrm{Nat}(\HH_V\times\XX,\YY)},\]
where $\HH_V\colon\fromto{T^{\op}}{\Top\subset\Cat_\infty}$ is the functor represented by $s$.

What is less clear is how to write down the cocartesian fibration classified by this internal Hom functor in a manner that avoids straightening and unstraightening. It turns out that rather than use the $T$-space $\underline{V}$ as a model for the functor $\HH_V$, it is more convenient to use the ``alternative'' undercategory $(T^{\op})^{V/}$ of \cite[\S 4.2.1]{HTT}.

The forgetful functor $u\colon\fromto{(T^{\op})^{V/}}{T^{\op}}$ is still a left fibration that is classified by the corepresentable functor $\HH_V$, so we are led to construct a cocartesian fibration
\[\fromto{\underline{\Fun}_{T^{\op}}(X,Y)}{T^{\op}}\]
whose fiber over an object $V\in T^{\op}$ is the $\infty$-category of functors $\fromto{(T^{\op})^{V/}\ultimes X}{Y}$ over $T^{\op}$ that carry $(u,p)$-cocartesian edges to $q$-cocartesian edges. This is exactly what we will do.
\end{rec}

\begin{dfn}\label{dfn:internalHom} Now suppose $C$ and $D$ two $T$-$\infty$-categories. Define a simplicial set
\[\underline{\Fun}'_{T^{\op}}(C,D)\]
over $T^{\op}$ via the following universal property. For any map $\eta\colon\fromto{K}{T^{\op}}$, one demands a bijection
\[\Mor_{T^{\op}}(K,\underline{\Fun}'_{T^{\op}}(C,D))\cong\Mor_{K\downarrow_{T^{\op}}T^{\op}}(K\downarrow_{T^{\op}}C,K\downarrow_{T^{\op}}D),\]
natural in $\eta$. (Here we are using the notation of the lax pullback from Nt. \ref{ntn:laxpb}.)

Now let us define
\[\underline{\Fun}_{T^{\op}}(C,D)\subset\underline{\Fun}'_{T^{\op}}(C,D)\]
as the full subcategory spanned by those objects
\[\fromto{({T^{\op}})^{V/}\times_{T^{\op}}C\cong\{V\}\downarrow_{T^{\op}}C}{\{V\}\downarrow_{T^{\op}}D\cong {(T^{\op})}^{V/}\times_{T^{\op}}D}\]
that carry cocartesian edges over $({T^{\op}})^{V/}$ to cocartesian edges over $({T^{\op}})^{V/}$.
\end{dfn}

\begin{nul} We will soon show that $\underline{\Fun}_{T^{\op}}(C,D)$ is a $T$-$\infty$-category. The fibre of this $T$-$\infty$-category over any object $V\in T$ is by construction equivalent to the $\infty$-category of $\underline{V}^{\op}$-functors
\[\fromto{C\ultimes\underline{V}}{D\ultimes\underline{V}}\]
\end{nul}

\begin{ntn} If $\fromto{M}{S}$ and $\fromto{N}{S}$ are two maps of simplicial sets, and if $E\subset M_1$ and $F\subset N_1$ are collections of marked edges (which as usual we assume contain all degenerate edges), then we may write
\[(M,E)\downarrow_S(N,F)\coloneq (M,E)\underset{\Fun(\Delta^{\{0\}},S)^{\sharp}}{\times}\Fun(\Delta^1,S)^{\sharp}\underset{\Fun(\Delta^{\{1\}},S)^{\sharp}}{\times}(N,F).\]
\end{ntn}

\begin{nul} With this notation in hand, if $C$ and $D$ are two $T$-$\infty$-categories, we may describe $\fromto{\underline{\Fun}_{T^{\op}}(C,D)}{T^{\op}}$ via the following universal property: for any map $\fromto{K}{T^{\op}}$, we demand a bijection
\[\Mor_{/T^{\op}}(K,\underline{\Fun}_{T^{\op}}(C,D))\cong\Mor_{/T^{\op}}(K^{\flat}\downarrow_{T^{\op}}\leftnat{C},\leftnat{D}),\]
functorial in $\sigma$.
\end{nul}

\begin{prp}\label{prp:GFuncat} Suppose $C$ and $D$ two $T$-$\infty$-categories. Then the restriction
\[\underline{\Fun}_{T^{\op}}(C,D) \to T^{\op}\]
of the structure map above is a cocartesian fibration, and an edge $e$ is marked if and only if the corresponding map
\[\fromto{\Delta^1 \downarrow_{T^{\op}} C}{\Delta^1 \downarrow_{T^{\op}} D}\]
over $\Delta^1 \downarrow_{T^{\op}} T^{\op}$ carries cocartesian edges to cocartesian edges.
\begin{proof} Let us write $S\coloneq T^{\op}$. The map $r\colon\fromto{\underline{\Fun}_S(C,D)}{S}$ can be described as follows: for any map of simplicial sets $\fromto{K}{S}$, one has a bijection
\[\Mor_S(K,\underline{\Fun}_S(C,D))\cong\Mor_S(K^{\flat}\downarrow_S\leftnat{C},\leftnat{D}).\]

We see that the data of the source map
\[s\colon\fromto{S\downarrow_SC}{S},\]
the target map
\[t\colon\fromto{S\downarrow_SC}{S},\]
and the marking on $S^{\sharp}\downarrow_S\leftnat{C}$ together 
enjoy the conditions of \cite[Cor. 6.2.1]{BarShah}. Indeed, it is a trivial matter to see that equivalences of $S\downarrow_SC$ are marked and that marked edges are closed under composition. By \cite[Cor. 2.4.7.12]{HTT}, the source map $s$ is a cartesian fibration, and that a map is $s$-cartesian if and only if its projection to $C$ is an equivalence. Consequently, since the marked edges are those whose projection to $C$ is $p$-cocartesian, it follows that for any $2$-simplex
\begin{equation*}
\begin{tikzpicture}[baseline]
\matrix(m)[matrix of math nodes,
row sep=4ex, column sep=4ex,
text height=1.5ex, text depth=0.25ex]
{&y&\\
x&&z\\};
\path[>=stealth,->,font=\scriptsize]
(m-1-2) edge[inner sep=0.75pt] node[above right]{$\psi$} (m-2-3)
(m-2-1) edge[inner sep=0.75pt] node[above left]{$\phi$} (m-1-2)
edge node[below]{$\chi$} (m-2-3);
\end{tikzpicture}
\end{equation*}
in which $\psi$ is $s$-cartesian, the edge $\phi$ is marked just in case the edge $\chi$ is so. This proves that the conditions of \cite[Cor. 6.2.1]{BarShah} indeed apply.

As a consequence, we deduce that $r$ is a cocartesian fibration, and that an edge of $\underline{\Fun}_S(C,D)$ is $r$-cocartesian just in case the corresponding functor $\fromto{\Delta^1\downarrow_SD}{\Delta^1\downarrow_SD}$ over $\Delta^1\downarrow_SS$ carries cocartesian edges to cocartesian edges, as desired.
\end{proof}
\end{prp}

We now set about proving that $\underline{\Fun}_{T^{\op}}$ really is an internal hom.

\begin{thm}\label{thm:internalHom} Suppose $C$, $D$, and $E$ three $T$-$\infty$-categories.
\begin{enumerate}
\item There is a natural equivalence of $T$-$\infty$-categories
\[\equivto{\underline{\Fun}_{T^{\op}}(C,\underline{\Fun}_{T^{\op}}(D,E))}{\underline{\Fun}_{T^{\op}}(C \ultimes D, E)}.\]
\item Accordingly, there is a natural equivalence of $\infty$-categories
\[ \equivto{\Fun_{T^\op}(C,\underline{\Fun}_{T^\op}(D,E))}{\Fun_{T^\op}(C \ultimes D,E)}. \]
\end{enumerate}
\end{thm}

\noindent To prove this result, we will need a lemma.

\begin{lem}\label{lem:inflatingByArrowCat} Let $S$ be an $\infty$-category, $\iota\colon S \to \mathscr{O}(S)$ be the identity section and regard $\mathscr{O}(S)^\sharp$ as a marked simplicial set over $S$ via the target map. Then
\begin{enumerate}
\item For every marked simplicial set $X \to S$ and cartesian fibra{}tion $C \to S$,
\begin{equation*}
\id_X \times \iota \times \id_C: X \times_S \rightnat{C} \to X \downarrow_S \rightnat{C}
\end{equation*}
is a cocartesian equivalence in $s\Set^+_{/S}$.
\item For every marked simplicial set $X \to S$ and cocartesian fibration $C \to S$,
\begin{equation*}
\id_C \times \iota \times \id_X : \leftnat{C} \times_S X \to \leftnat{C} \downarrow_S X
\end{equation*}
is a homotopy equivalence in $s\Set^+_{/S}$.
\end{enumerate}
\begin{proof} For (1), using \cite[Cor. 6.2.1]{BarShah} on the span $S \ot \rightnat{C} \to S$, we reduce to the case where $C = S$. By definition, $X \to X \downarrow_SS^{\sharp}$ is a cocartesian equivalence if and only if for every cocartesian fibration $Z \to S$, the map
\[\Map^\sharp_S(X \downarrow_SS^{\sharp}, \leftnat{Z}) \to \Map^\sharp_S(X,\leftnat{Z})\]
is a trivial Kan fibration. In other words, for every monomorphism of simplicial sets $A \to B$ and cocartesian fibration $Z \to S$, we need to provide a dotted lift in the following commutative square
\begin{equation*}
\begin{tikzpicture}[baseline]
\matrix(m)[matrix of math nodes,
row sep=4ex, column sep=4ex,
text height=1.5ex, text depth=0.25ex]
 { (B^\sharp \times X) \cup^{(A^\sharp \times X)} ((A^\sharp \times X) \downarrow_S S^{\sharp}) & \leftnat{Z} \\
 (B^\sharp \times X) \downarrow_S S^{\sharp} & S^{\sharp}. \\ };
\path[>=stealth,->,font=\scriptsize]
(m-1-1) edge node[above]{$\phi$} (m-1-2)
edge (m-2-1)
(m-1-2) edge (m-2-2)
(m-2-1) edge (m-2-2)
edge[dotted] (m-1-2);
\end{tikzpicture}
\end{equation*}

Define $h_0\colon\mathscr{O}(S)^\sharp \times (\Delta^1)^\sharp \to \mathscr{O}(S)^\sharp$ as the adjoint to the map $\mathscr{O}(S)^\sharp \to \mathscr{O}(\mathscr{O}(S))^\sharp$ obtained by precomposing by the functor $\min\colon\Delta^1 \times \Delta^1 \to \Delta^1$, so that $(1,1)$ is the unique vertex sent to $1$. Precomposing $\phi$ by $\id_{A^\sharp \times X} \times h_0$, define a homotopy
\begin{equation*}
h\colon (A^\sharp \times X) \downarrow_S S^{\sharp} \times (\Delta^1)^\sharp \to \leftnat{Z}
\end{equation*}
from $\phi|_{A^\sharp \times X} \circ \pr_{A^\sharp \times X}$ to $\phi|_{(A^\sharp \times X) \downarrow_S S^\sharp}$. Using $h$ and $\phi|_{B^\sharp \times X}$, define a map
\begin{equation*}
\psi\colon (B^\sharp \times X) \cup^{(A^\sharp \times X)} ((A^\sharp \times X) \downarrow_S S^{\sharp}) \to \Fun((\Delta^1)^\sharp, \leftnat{Z})
\end{equation*}
such that $\psi|_{B^\sharp \times X}$ is adjoint to $\phi|_{B^\sharp \times X} \circ \pr_{B^\sharp \times X}$ and $\psi|_{(A^\sharp \times X) \downarrow_S S^{\sharp}}$ is adjoint to $h$. Then we may factor the above square through the trivial fibration $\Fun((\Delta^1)^\sharp, \leftnat{Z}) \to \leftnat{Z} \downarrow_S S^{\sharp}$ of Lm. \ref{push} to obtain the commutative rectangle
\begin{equation*}
\begin{tikzpicture}[baseline]
\matrix(m)[matrix of math nodes,
row sep=4ex, column sep=4ex,
text height=1.5ex, text depth=0.25ex]
 { (B^\sharp \times X) \cup^{(A^\sharp \times X)} ((A^\sharp \times X) \downarrow_S S^{\sharp}) & \Fun((\Delta^1)^\sharp, \leftnat{Z}) &  \leftnat{Z} \\
 (B^\sharp \times X) \downarrow_S S^{\sharp} & \leftnat{Z} \downarrow_S S^{\sharp} & S^\sharp \\ };
\path[>=stealth,->,font=\scriptsize]
(m-1-1) edge node[above]{$\psi$} (m-1-2)
edge (m-2-1)
(m-1-2) edge node[above]{$e_1$} (m-1-3)
edge[->>] node[left]{$\sim$} (m-2-2)
(m-1-3) edge (m-2-3)
(m-2-1) edge node[below]{$\phi|_{B^\sharp \times X} \times \id$} (m-2-2)
edge[dotted] node[below]{$\widetilde{\psi}$} (m-1-2)
(m-2-2) edge node[below]{$e_1$} (m-2-3);
\end{tikzpicture}
\end{equation*}
The dotted lift $\widetilde{\psi}$ exists, and $e_1 \circ \widetilde{\psi}$ is our desired lift.

The latter statement follows readily from the proof of Lm. \ref{lem:Ppinverseofiota}, once one notes that the homotopy inverse and the relevant homotopies all respect the markings in (2).
\end{proof}
\end{lem}

Armed with this, the result is no trouble to prove.
\begin{proof}[Proof of Th. \ref{thm:internalHom}] Set $S \coloneq T^\op$. For (1), consider the commutative diagram
\[ \begin{tikzpicture}[baseline]
\matrix(m)[matrix of math nodes,
row sep=6ex, column sep=4ex,
text height=1.5ex, text depth=0.25ex]
 { S^\sharp \downarrow_S \leftnat{(C \ultimes D)} \\
 & S^\sharp \downarrow_S \leftnat{C} \downarrow_S \leftnat{D} & S^\sharp \downarrow_S \leftnat{D} & S^\sharp \\
   & S^\sharp \downarrow_S \leftnat{C} & S^\sharp \\
   & S^\sharp \\ };
\path[>=stealth,->,font=\scriptsize]
(m-1-1) edge node[below]{$f$} (m-2-2)
edge (m-4-2)
edge (m-2-4)
(m-2-2) edge (m-2-3)
edge (m-3-2)
(m-2-3) edge (m-2-4)
edge (m-3-3)
(m-3-2) edge (m-3-3)
edge (m-4-2);
\end{tikzpicture} \]
where the square is a pullback square. By Lm. \ref{lem:inflatingByArrowCat}(2), the map 
\[ \leftnat{(\mathscr{O}(S) \times_S C)} \times_S \leftnat{D} \to \leftnat{(\mathscr{O}(S) \times_S C)} \downarrow_S \leftnat{D}\]
is a homotopy equivalence in $s\Set^+_{/S}$ (via the target map to $S$), hence so is $f$, because the additional marked edges do not obstruct the data of the homotopy equivalence. Invoking \cite[Lm. 6.6]{BarShah} and \cite[Lm. 6.7]{BarShah}, we conclude. For (2), we may simply repeat the same argument with the leftmost copy of $S$ removed.
\end{proof}

\begin{exm}\label{exm:Tpresheaves} One of the most important examples, not surprisingly, is the \emph{$T$-presheaf} $T$-$\infty$-category
\[\underline{\PP}(C)\coloneq\underline{\Fun}_{T^{\op}}(C^{\vop},\underline{\Top}_T),\]
defined for any $T$-$\infty$-category $C$. In the next section, we will provide a Yoneda embedding, and we will show that it is fully faithful.
\end{exm}

\begin{wrn} The construction $\underline{\Fun}_{T^\op}(-,-)$ does not make homotopical sense when the first variable is not fibrant, so it does not yield a Quillen bifunctor on $s\Set^+_{/T^\op}$. Nevertheless, we can say the following about varying the first variable.
\end{wrn}

\begin{prp} Let $C$, $D$, and $E$ be $T$-$\infty$-categories, and let $F\colon C \to D$ be a $T$-functor. Denote by
\[ F^{\star}\colon \underline{\Fun}_{T^\op}(D,E) \to \underline{\Fun}_{T^\op}(C,E) \]
denote the $T$-functor induced by $F$.
\begin{enumerate}
\item If $F$ is an equivalence, then so is $F^{\star}$.
\item If $F$ is a cofibration, then $F^{\star}$ is a fibration in $s\Set^+_{/S}$.
\end{enumerate}
\begin{proof} Let $S = T^\op$. For (1), since $F\colon \leftnat{C} \to \leftnat{D}$ is a homotopy equivalence in $s\Set^+_{/S}$, so is $S^{\sharp}\downarrow_S \leftnat{C} \to S^{\sharp} \downarrow_S \leftnat{D}$. Thus it follows from \cite[Lm. 6.7]{BarShah} that $F^{\star}$ is an equivalence.

For (2), we need to verify that for any trivial cofibration $A \to B$ in $s\Set^+_{/S}$, the map
\[ (A \downarrow_S \leftnat{D}) \cup^{(A \downarrow_S \leftnat{C})} (B \downarrow_S \leftnat{C}) \to B \downarrow_S \leftnat{D} \]
is a trivial cofibration in $s\Set^+_{/S}$. By the proof of Pr. \ref{prp:GFuncat}, $- \downarrow_S \leftnat{C}$ preserves trivial cofibrations and ditto for $\leftnat{D}$. The result then follows.
\end{proof}
\end{prp}


\section{The parametrized Yoneda embedding} In \cite[\S 5]{BGN}, we have already constructed a parametrized form of the twisted arrow $\infty$-category of a $T$-$\infty$-category. Let us recall elements of this construction.

\begin{rec} Suppose $T$ an $\infty$-category, and suppose $(C,p)$ a $T$-$\infty$-category. In \cite[5.5]{BGN} we define $\widetilde{\mathscr{O}}(C/T^{\op})$ so that the objects are morphisms $f\colon\fromto{u}{v}$ of $C$ such that $p(f)$ is an identity morphism in $T$, and a morphism $\fromto{f}{g}$ is a commutative diagram
\begin{equation*}
\begin{tikzpicture}[baseline]
\matrix(m)[matrix of math nodes, 
row sep=3ex, column sep=4ex, 
text height=1.5ex, text depth=0.25ex] 
{u&&x\\[-4ex]
&w&\\ 
v&&y\\}; 
\path[>=stealth,->,font=\scriptsize] 
(m-1-3) edge node[above]{$\psi$} (m-2-2)
edge node[right]{$g$} (m-3-3)
(m-1-1) edge node[above]{$\phi$} (m-2-2) 
edge node[left]{$f$} (m-3-1) 
(m-2-2) edge (m-3-3) 
(m-3-1) edge node[below]{$\xi$} (m-3-3); 
\end{tikzpicture}
\end{equation*}
in which $\phi$ is $p$-cocartesian, and $p(\psi)$ is an identity morphism. Composition is performed by forming the pushout along the top and simple composition along the bottom.

The functor
\[M\colon\fromto{\widetilde{\mathscr{O}}(C/T^{\op})}{C^{\vop}\ultimes C}\]
will carry an object $f\in\widetilde{\mathscr{O}}(C/T^{\op})$ as above to the pair of objects $(u,v)\in C^{\vop}\ultimes C$, and it will carry a morphism $\fromto{f}{g}$ as above to the pair of morphisms
\begin{equation*}
\left(\begin{tikzpicture}[baseline]
\matrix(m)[matrix of math nodes, 
row sep=2ex, column sep=3ex, 
text height=1.5ex, text depth=0.25ex] 
{&w&\\ 
u&&x\\}; 
\path[>=stealth,<-,inner sep=0.9pt,font=\scriptsize] 
(m-1-2) edge node[above left]{$\phi$} (m-2-1) 
edge node[above right]{$\psi$} (m-2-3); 
\end{tikzpicture}
,\ v\ \tikz[baseline]\draw[>=stealth,->,font=\scriptsize](0,0.5ex)--node[above]{$\xi$}(0.5,0.5ex);\ y\right)\in C^{\vop}\ultimes C.
\end{equation*}
We demonstrate that $M$ is a left $T$-fibration.
\end{rec}

\begin{dfn} In light of Pr. \ref{prp:Tstraighten}, the left $T$-fibration $M$ is classified by a $T$-functor
\[\underline{\Mor}_C\colon\fromto{C^{\vop}\ultimes C}{\underline{\Top}_T},\]
and therefore by Th. \ref{thm:internalHom} also a $T$-functor
\[j\colon\fromto{C}{\underline{\PP}(C)},\]
called the \emph{parametrized Yoneda embedding.}
\end{dfn}

In order to justify this bit of terminology, we should actually demonstrate that $j$ is, in fact, fully faithful. To verify this, we work fibrewise: for any object $V\in T$, we look at the functor on fibers
\[j_V\colon\fromto{C_V}{\Fun_{\underline{V}}(C^{\vop}\ultimes \underline{V}, \underline{\Top}_{\underline{V}^{\op}})}.\]
Our universal property of $T$-spaces (Th. \ref{thm:univpropDG}) now permits one to regard the target of $j_V$ as the $\infty$-category $\Fun(C^{\vop}\times_{T^{\op}} \underline{V}, \Top)$. Unwinding the definitions above, it is easy to see that $j_V$ factors up to homotopy as the inclusion $\into{C_V}{(C^{\vop}\times_{T^{\op}}\underline{V})^{\op}}$ followed by Yoneda embedding
\[\into{(C^{\vop}\times_{T^{\op}}\underline{V})^{\op}}{\Fun(C^{\vop}\times_{T^{\op}} \underline{V}, \Top)}.\]
This already proves that $j_V$ is fully faithful, but in fact we also can conclude the following.
\begin{prp} Given $T$, $C$, and $V$ as above, for any object $X\in C_V$, and for any $T$-functor
\[F\colon\fromto{C^{\vop}\ultimes \underline{V}}{\underline{\Top}_{\underline{V}}}\]
we have an equivalence
\[
F(X)\simeq\Map_{\underline{\Fun}_{\underline{V}}(C^{\vop}\ultimes\underline{V},\underline{\Top}_{\underline{V}^{\op}})}(j_V(X),F).
\]
\end{prp}
\noindent In any case, we immediately deduce the following.
\begin{thm}[Parametrized Yoneda lemma] For any $\infty$-cate\-gory $T$, and for any $T$-$\infty$-cate\-gory $C$, the parametrized Yoneda embedding
\[j\colon\fromto{C}{\underline{\PP}(C)}\]
is fully faithful.
\end{thm}


\appendix

\section{Notational glossary}
\begin{longtable}{lp{0.8\linewidth}}
$C^\sim$ & the marked simplicial set $(C,iC)$ attached to an $\infty$-category $C$ in which $iC\subset C_1$ is the class of equivalences; equivalently, $C^{\sim}=\rightnat{C}=\leftnat{C}$ for the cartesian and cocartesian fibration $\fromto{C}{\Delta^0}$.\\
$C^{\vop}$ & the \emph{opposite $T$-$\infty$-category} $C^{\vop}=(C^{\op})^{\vee}=({}_\vee C)^{\op}$ of a $T$-$\infty$-category $C$; see Df. \ref{dfn:opposite}.\\
$\Cat_{\infty}$ & the $\infty$-category of $\infty$-categories, i.e., the simplicial nerve of the simplicial category of fibrant marked simplicial sets; see \cite[Df. 3.0.0.1]{HTT}.\\
$\Cat_{\infty,T}$ & the $\infty$-category of $T$-$\infty$-categories, i.e., the simplicial nerve of the simplicial category of fibrant marked simplicial sets over $T^{\op}$ in the cocartesian model structure; see Nt. \ref{ntn:CatinftyT}.\\
$\FF_G$ & the nerve of the ordinary category of all finite, continuous $G$-sets and $G$-equivariant maps; see Nt. \ref{ntn:finiteTsets}.\\
$\FF_T$ & the \emph{$\infty$-category of finite $T$-sets}, i.e., the full subcategory of $\Fun(T^{\op},\Top)$ spanned by the finite corpoducts of representables; see Nt. \ref{ntn:finiteTsets}.\\
$\underline{\FF}_T$ & the \emph{$T$-$\infty$-category of finite $T$-sets}; see Ex. \ref{exm:tCatoffiniteTsets}.\\
$\Fun_{T^{\op}}$ & the $\infty$-category of functors between two $\infty$-categories; see Nt. \ref{ntn:CatinftyT}.\\
$\underline{\widetilde{\Fun}}_{T^{\op}}$ & the $T$-$\infty$-category of functors from the left dual of a $T^{\op}$-$\infty$-category to a $T$-$\infty$-category; see Cnstr. \ref{cor32213}.\\
$\underline{\Fun}_{T^{\op}}$ & the $T$-$\infty$-category of functors between two $T$-$\infty$-categories; see Df. \ref{dfn:internalHom}.\\
$G$ & a profinite group.\\
$\Map_{T^{\op}}$ & the \emph{space of $T$-functors} between two $T$-$\infty$-categories; see Nt. \ref{ntn:CatinftyT}.\\
$M\downarrow_S N$ & the \emph{lax pullback} of two maps $\fromto{M}{S}$ and $\fromto{N}{S}$; see Nt. \ref{ntn:laxpb}.\\
$\mathscr{O}(C)$ & the \emph{arrow $\infty$-category} $\Fun(\Delta^1,C)$ of an $\infty$-category $C$.\\
$\widetilde{\mathscr{O}}(C)$ & the \emph{twisted arrow $\infty$-category} of an $\infty$-category $C$, so that the $n$-simplices of $\widetilde{\mathscr{O}}(C)$ are functors $\fromto{(\Delta^n)^\op\star\Delta^n}{C}$; see \cite[Nt. 2.3]{M1}, \cite[Df. 2.1]{BGN}, and \cite[\S 1]{Aefffib}.\\
$\OO_G$ & the full subcategory of $\FF_{G}$ spanned by the finite \emph{$G$-orbits}, i.e., the finite transitive (continuous) $G$-sets; see Df. \ref{dfn:Gcat}.\\
$\categ{Orbit}(U)$ & the set of transitive $G$-subsets $W\subset U$ of a $G$-set $U$, or, more generally, the set of representable summands of a finite $T$-set $U$; see Nt. \ref{ntn:finiteTsets}.\\
$\underline{\PP}(C)$ & the $T$-$\infty$-category of $T$-presheaves on a $T$-$\infty$-category $C$; see Ex. \ref{exm:Tpresheaves}.\\
$\Top$ & the $\infty$-category of spaces, i.e., the $\infty$-category freely generated under colimits by one object, or, equivalently, the simplicial nerve of the fibrant simplicial category of Kan complexes, or, again equivalently, the full subcategory of $\Cat_{\infty}$ spanned by those objects in which every edge is marked; see \cite[Df. 1.2.16.1]{HTT}.\\
$\underline{U}$ & the \emph{discrete $T$-space} attached to a finite $T$-set $U$; see Ex. \ref{exm:representable} and Ex. \ref{discreteGcat}.\\
$\rightnat{X}$ & the marked simplicial set $(X,i^SX)$ attached to a cartesian fibration $p\colon\fromto{X}{S}$ in which $i^SX\subset X_1$ is the class of $p$-cartesian edges; see \cite[Df. 3.1.1.8]{HTT}.\\
$X^{\vee}$ & the \emph{right dual cocartesian fibration} $p^{\vee}\colon\fromto{X^{\vee}}{S^{\op}}$ of a cartesian fibration $p\colon\fromto{X}{S}$, i.e., the simplicial set whose $n$-simplices are functors $x\colon\fromto{\widetilde{\mathscr{O}}(\Delta^n)^{\op}}{X}$ such that any morphism of the form $\fromto{x_{ik}}{x_{jk}}$ lies over an identity in $S$, and any morphism of the form $\fromto{x_{ij}}{x_{ik}}$ is $p$-cartesian; see \cite{BGN} and \cite[Cnstr. 2.6]{Aefffib}.\\
$\leftnat{Y}$ & the marked simplicial set $(Y,i_TY)$ attached to a cocartesian fibration $q\colon\fromto{Y}{T}$ in which $i_TY\subset Y_1$ is the class of $q$-cocartesian edges; see \cite[Df. 3.1.1.8]{HTT}.\\
${}_{\vee}Y$ & the \emph{dual cartesian fibration} ${}_{\vee}q\coloneq((q^{\op})^{\vee})^{\op}\colon\fromto{{}_{\vee}Y\coloneq((X^{\op})^{\vee})^{\op}}{T^{\op}}$ of a cocartesian fibration $q\colon\fromto{Y}{T}$; see \cite{BGN} and \cite[Cnstr. 2.6]{Aefffib}.\\
\end{longtable}


\bibliographystyle{amsplain}
\bibliography{Gcats}

\end{document}